\documentclass[12pt]{article}
\usepackage{amsfonts,amssymb,amsmath,titling,titlesec,cancel,amsthm,color, graphicx,hyperref}
\usepackage{verbatim,tikz}
\usetikzlibrary{arrows}
\usepackage{epstopdf}
\usepackage{graphicx}
\usepackage{pifont}
\usepackage{color}
\usepackage{tkz-graph}
\usepackage{caption,subcaption}
\usepackage{hyperref}

\AppendGraphicsExtensions{.gif}

\titleformat{\section}{\large\bfseries}{\thesection.}{1em}{}

\newtheorem{theorem}{Theorem}[section]
\newtheorem{lemma}[theorem]{Lemma}

\newtheorem{corollary}[theorem]{Corollary}

\newtheorem{conjecture}[theorem]{Conjecture}

\title{The $m$-Degenerate Chromatic Number of a Digraph}
\author{Noah Golowich\thanks{Research Science Institute, Massachusetts Institute of Technology, Cambridge, MA}}
\date{January 1, 2016}

\begin{document}
\maketitle





\begin{abstract}
The {\it digraph chromatic number} of a directed graph $D$, denoted $\chi_A(D)$, is the minimum positive integer $k$ such that there exists a partition of the vertices of $D$ into $k$ disjoint sets, each of which induces an acyclic subgraph. For any $m \geq 1$, a digraph is {\it weakly $m$-degenerate} if each of its induced subgraphs has a vertex of in-degree or out-degree less than $m$. We introduce a generalization of the digraph chromatic number, namely $\chi_m(D)$, which is the minimum number of sets into which the vertices of a digraph $D$ can be partitioned so that each set induces a weakly $m$-degenerate subgraph. We show that for all digraphs $D$ without directed 2-cycles, $\chi_m(D) \leq \frac{2{\Delta}(D)}{4m+1} + O(1)$. Because $\chi_1(D) = \chi_A(D)$, we obtain as a corollary that $\chi_A(D) \leq 2/5 \cdot {\Delta}(D) + O(1)$. We then use this bound to show that $\chi_A(D) \leq \sqrt{2/3} \cdot \tilde{\Delta}(D) + O(1)$, substantially improving a bound of Harutyunyan and Mohar that states that $\chi_A(D) \leq (1 - e^{-13})\cdot \tilde\Delta(D)$ for large enough $\tilde\Delta(D)$.
\end{abstract}

\tikzset{
    every node/.style={
        circle,
        draw,
        solid,
        fill=black!50,
        inner sep=0pt,
        minimum width=4pt
    }
}

\section{Introduction}
A proper {\it vertex coloring} of an undirected graph partitions its vertices into independent sets. To extend this notion to directed graphs (digraphs), we consider acyclic sets instead of independent sets. An {\it acyclic set} in a digraph is a set of vertices whose induced subgraph contains no directed cycle. The {\it digraph chromatic number} of a digraph $D$, denoted $\chi_A(D)$, is then defined to be the minimum number of acyclic sets into which the vertices of $D$ can be partitioned. The digraph chromatic number was originally defined by Neumann-Lara in the early 1980's \cite{neumann_vertex_1985}. In this paper, we primarily consider {\it oriented graphs}, which are digraphs such that at most one edge connects any pair of vertices.

Many upper bounds on the chromatic number of undirected graphs are phrased in terms of $\Delta(G)$, the maximum degree of $G$.
To extend this notion to directed graphs, there are a few options which measure the maximum degree. Given a digraph $D$, $\tilde{\Delta}(D)$ is the maximum geometric mean of the in-degree and the out-degree of a vertex in $D$, $\Delta(D)$ is the maximum total degree of a vertex in $D$,  and $\bar{\Delta}(D)$ is the maximum arithmetic mean of the in-degree and the out-degree of a vertex in $D$. Notice that for any digraph $D$, we have $ \tilde{\Delta}(D) \leq \bar{\Delta}(D) = \Delta(D)/2$.

The digraph chromatic number $\chi_A(D)$ is one of many chromatic numbers which have been defined for digraphs. Bokal et al.~\cite{circular_bokal_2004} introduced the circular chromatic number of a digraph $D$, denoted $\chi_c(D)$, as a generalization of the digraph chromatic number. Let $S_p$ denote the circle with perimeter $p$, and for $x, y \in S_p$, let $d(x, y)$ denote the clockwise distance from $x$ to $y$. Then the circular chromatic number $\chi_c(D)$ is the infimum of all positive real numbers $p$ for which there exists a function $c: V(D) \rightarrow S_p$ such that for each edge $uv$ of $D$, $d(c(u), c(v)) \geq 1.$ Bokal et al.~showed that $\chi_c(D)$ takes on rational values, and moreover, that $\chi_A(D) - 1 < \chi_c(D) \leq \chi_A(D)$. 

The digraph chromatic number has received the most attention among colorings of directed graphs because recent results \cite{acyclic_aharoni_2008, ben_size_2012, circular_bokal_2004, chen_coloring_2014, two_harut_2012, digraph_keevash_2013} suggest that the digraph chromatic number in digraphs behaves similarly to the chromatic number in undirected graphs. 
Much still remains to be learned however. For instance, it is easily proved using the greedy algorithm that $\chi_A(D) \leq \lfloor \tilde{\Delta}(D) \rfloor + 1 \leq \lfloor {\Delta}(D)/2 \rfloor + 1$; this is analogous to the fact that in an undirected graph $G$, $\chi(G) \leq \Delta(G) + 1$. However, this bound is not tight for most digraphs.

In the case of undirected graphs, Brooks \cite{brooks_coloring_1941} made the first improvement on the obvious bound of $\chi(G) \leq \Delta(G) + 1$; he showed that $\chi(G) \leq \Delta(G)$ unless $G$ is a complete graph or an odd cycle. Borodin and Kostochka \cite{borodin_upper_1977} and Catlin \cite{catlin_bound_1978} then independently strengthened Brooks' theorem, showing that if $G$ is triangle-free, then $\chi(G) \leq \left\lceil\frac{3(\Delta(G) + 1)}{4}\right\rceil$. Mohar \cite{mohar_eigen_2009} recently proved an analogue of Brooks' theorem for all digraphs. It follows from Mohar's results that if $D$ is an oriented graph with $\Delta(D) > 2$, then $\chi_A(D) \leq \lceil{\Delta}(D)/2\rceil$.
However, it appears that this upper bound on $\chi_A(D)$ can be significantly improved further; Harutyunyan and Mohar \cite{two_harut_2012} credit McDiarmid and Mohar with the following conjecture.
\begin{conjecture}[McDiarmid \& Mohar, 2002 \cite{two_harut_2012}]
\label{conj:everything}
Every oriented graph $D$ satisfies {$\chi_A(D) = O\left(\frac{{\Delta}(D)}{\log {\Delta}(D)}\right)$}.
\end{conjecture}
Conjecture \ref{conj:everything} is analogous to a result for undirected graphs by Kim \cite{kim_brooks_1995}, who showed that \mbox{$\chi(G) \leq (1 + o(1)) \frac{\Delta(G)}{\log \Delta(G)}$} if the girth (length of the shortest cycle) of $G$ is greater than 4. Johansson \cite{johannson_asymptotic} later extended Kim's bound to graphs $G$ of girth greater than 3, and Jamall \cite{jamall_brooks_2011} has since used a simpler proof to strengthen Johansson's bound by a constant factor. Kim, Johansson, and Jamall all used the probabilistic method to prove upper bounds on the chromatic number. 

Harutyunyan and Mohar \cite{strengthened_harut_2011} applied the probabilistic method to digraphs to show that $\chi_A(D) \leq (1 - e^{-13})\tilde{\Delta}(D)$ for $\tilde\Delta(D)$ large enough, which only slightly improves upon the trivial bound of $\chi_A(D) \leq \lfloor \tilde{\Delta}(D) \rfloor + 1$ and is far from the bound given in Conjecture \ref{conj:everything}. They also posed the following related conjecture, which, although much weaker than Conjecture \ref{conj:everything} for $\tilde{\Delta}(D)$ large, gives a precise bound for all $\tilde{\Delta}(D)$, unlike Conjecture \ref{conj:everything}. 
\begin{conjecture}[Harutyunyan \& Mohar, 2011 \cite{strengthened_harut_2011}]
\label{conj:harutdelta}
Let $D$ be an oriented graph. Then $\chi_A(D) \leq \left\lceil \tilde{\Delta}(D)/2 \right\rceil + 1.$
\end{conjecture}

In this paper, we use the following generalization of the digraph chromatic number and deduce new bounds on the digraph chromatic number itself as a special case of our results. For a positive integer $m$, a digraph $D$ is said to be {\it weakly $m$-degenerate} if for every induced subgraph of $D$, there is a vertex of out-degree or in-degree strictly less than $m$. Therefore, a digraph is weakly 1-degenerate if and only if it is acyclic. Given a positive integer $k$, a {\it $(k, m)$-degenerate coloring} of $D$ is a partition of $V(D)$ into $k$ sets, each of which induces a weakly $m$-degenerate subgraph. More generally, an {\it$m$-degenerate coloring} of $D$ is a partition of $V(D)$ into some number of sets, each of which induces an $m$-degenerate subgraph. Given a positive integer $m$, we let $\chi_m(D)$, the {\it $m$-degenerate chromatic number} of $D$, be the smallest positive integer $k$ such that $D$ has a $(k, m)$-degenerate coloring. Notice that $\chi_1(D) = \chi_A(D)$; hence the parameter $\chi_m(D)$ is a generalization of $\chi_A(D)$. Bokal et al.~\cite{circular_bokal_2004} showed some further connections between weak degeneracy and digraph colorings. For instance, if a digraph is weakly $m$-degenerate, then $\chi_A(D) \leq m+1$; moreover, this bound is tight for each positive integer $m$.

In Theorem \ref{thm:fracdelta}, we prove an upper bound on the $m$-degenerate chromatic number of any digraph $D$ in terms of ${\Delta}(D)$. 

\begin{theorem}
\label{thm:fracdelta}
Let $m$ be a positive integer. For any oriented graph $D$, we have
\begin{equation}
\chi_m(D)  \leq \left\lfloor \frac{{\Delta}(D)  -  \left\lfloor \frac{{\Delta}(D) + 1}{4m + 1} \right\rfloor}{2m}\right\rfloor + 1\nonumber.
\end{equation}
\end{theorem}

By taking $m = 1$, the following corollary follows immediately from Theorem \ref{thm:fracdelta}.

\begin{corollary}
\label{cor:acyclic}
If $D$ is any oriented graph, $\chi_A(D) \leq \left\lfloor {2}/{5}\cdot({\Delta}(D) + 1) \right\rfloor + 1.$
\end{corollary}

Using Theorem \ref{thm:fracdelta}, we then find an upper bound on $\chi_m(D)$ in terms of $\tilde\Delta(D)$ (see Theorem \ref{cor:mdegtilde}). In the $m = 1$ case, we obtain the below corollary, which significantly improves Harutyunyan and Mohar's bound of $\chi_A(D) \leq (1-e^{-13}) \cdot \tilde{\Delta}(D)$ and makes progress towards Conjecture \ref{conj:harutdelta}.

\begin{corollary}
\label{cor:improvedbound}
For any oriented graph $D$, we have $\chi_A(D) \leq \left\lfloor \sqrt{\frac{2}{3}} \cdot \tilde{\Delta}(D) + \frac{7}{5} \right\rfloor$.
\end{corollary}

The organization of this paper is as follows. In Section \ref{sec:colorings}, we prove Theorem \ref{thm:fracdelta}, giving an upper bound on $\chi_m(D)$ for any digraph $D$ in terms of ${\Delta}(D)$. In Section \ref{sec:betbounds}, we improve upon this bound for $m= 1$ for a particular class of digraphs. We give some concluding remarks in the final section.

\section{Digraph colorings}
\label{sec:colorings}
Recall that Harutyunyan and Mohar \cite{strengthened_harut_2011} proved that given a digraph $D$, if $\tilde{\Delta}(D)$ is large enough, then $\chi_A(D) \leq (1 - e^{-13})\tilde{\Delta}(D)$. They used a non-constructive method to do so, and posed the problem of improving this bound, remarking that a different technique may be necessary. In this section, we use a constructive technique to prove Theorem \ref{thm:fracdelta}, which we then use to prove (Corollary \ref{cor:improvedbound}) the significantly stronger upper bound of $\chi_A(D) \leq \sqrt{2/3} \cdot \tilde{\Delta}(D) + O(1)$. 

It is easy to show that $\chi_m(D) \leq \left\lfloor \frac{{\Delta}(D)}{2m}\right \rfloor + 1$; the proof is similar to that of the fact that \mbox{$\chi_A(D) \leq \lfloor{\Delta}(D)/2\rfloor + 1$}. In particular, we color the vertices of $D$ greedily, in any order. At each step, the next vertex $v$ to be colored has either out-degree or in-degree at most ${\Delta}(D)/2$, suppose without loss of generality out-degree. Therefore, there are at most $\left \lfloor \frac{{\Delta}(D)}{2m} \right \rfloor$ colors which are present in at least $m$ out-neighbors of $v$. We now color $v$ using one of the remaining colors that is present in fewer than $m$ out-neighbors of $v$. The resulting coloring is indeed $m$-degenerate, since in any subset of any color class, the vertex in that subset colored last must have fewer than $m$ in-neighbors or out-neighbors in that subset.
Note that Theorem \ref{thm:fracdelta} improves the bound $\chi_m(D) \leq \left\lfloor \frac{{\Delta}(D)}{2m}\right \rfloor + 1$.

We prove Theorem \ref{thm:fracdelta} by using a strategy similar to one originally introduced independently by Borodin and Kostochka \cite{borodin_upper_1977} and by Catlin \cite{catlin_bound_1978} to prove an upper bound on the chromatic number in undirected graphs. The proof in the case $m = 1$ is quite short, so we provide a sketch of it here. Given an oriented graph $D$, by a theorem of Lov\'asz \cite{lovasz_on_1966} (Theorem \ref{thm:lovasz}), we can partition the vertices of $D$ into $s := \lfloor (\Delta(D) + 1)/5 \rfloor + 1$ sets, each inducing a digraph of maximum total degree of at most 4. If we can 2-color each of the resulting $s$ digraphs, then by using a different set of 2 colors for each one we obtain a digraph coloring of $D$ with at most $\frac{2}{5} \cdot \Delta(D) + O(1)$ colors.

To show that we can 2-color digraphs $D$ of maximum total degree at most 4, suppose not and let $D$ be a counter-example with a minimum number of vertices. (We will later define such a digraph to be (3,1)-critical.) If some vertex $v$ has in-degree less than 2, we can remove it, color the digraph induced by the remaining vertices of $D$ with 2 colors, and extend that coloring to $v$ by using the color not in the in-neighborhood of $v$. This produces a 2-coloring of $D$, contradicting our original assumption, so $D$ has no vertex with in-degree less than 2. By reversing all edges of $D$, we obtain that no vertex of $D$ has out-degree less than 2. Hence the in-degree and out-degree of each vertex of $D$ is exactly 2. Then by a theorem of Mohar \cite{mohar_eigen_2009} (Theorem \ref{thm:moharkcritical}), we obtain a contradiction to $D$ not being 2-colorable, so $D$ is in fact 2-colorable.

To prove Theorem \ref{thm:fracdelta} in its full generality, we follow an outline similar to the one described above. We begin by generalizing a directed graph analogue of Brooks' theorem \cite{brooks_coloring_1941} due to Mohar \cite{mohar_eigen_2009} to the framework of $m$-degenerate colorings. Our proof follows a similar outline to that of Mohar, who proved that any oriented graph $D$ with as few vertices as possible that satisfies $\chi_A(D) > \lceil \Delta(D)/2 \rceil$ must be a directed cycle. (Mohar in fact proved a more general result, classifying all such digraphs $D$ even if digons are allowed.) The key step of Mohar's proof is to find a degeneracy ordering of the vertices of $D$ which allows one to construct a coloring of $D$ with one fewer color than $\chi_A(D)$, unless $D$ is a directed cycle.

A few definitions are needed to state our next lemma. Given a digraph $D$ and $u \in V(D)$, we denote the subgraph induced on $V(D)\backslash \{u\}$ by $D - u$.  Similarly, for $X$ a set of vertices, $D - X$ is the subgraph induced on $V(D)\backslash X$.
Moreover, given a positive integer $m$, a {\it critical} vertex is a vertex $v \in V(D)$ such that $\chi_m(D - v) < \chi_m(D)$. If every vertex of $D$ is critical and $\chi_m(D) = k$, then we define $D$ to be a {\it $(k, m)$-critical} digraph. 
Our next lemma, Lemma \ref{lem:critvertex}, shows that critical vertices in a digraph must have large in-degree and out-degree.

\begin{lemma}
\label{lem:critvertex}
Suppose $v$ is a critical vertex in a digraph $D$, $m \geq 1$, and $\chi_m(D) = k$. Then $d^+(v) \geq (k-1)m$ and $d^-(v) \geq (k-1)m$.
\end{lemma}

\begin{proof}
Suppose for the purpose of contradiction that $d^+(v) < (k-1)m$. We will show that we can find a $(k-1, m)$-degenerate coloring of $D$, a contradiction to the fact that $\chi_m(D) = k$. Since $v$ is $(k, m)$-critical, we can find a $(k-1, m)$-degenerate coloring of $D - v$.  At least one color $c$ must be present in fewer than $m$ out-neighbors of $v$ because otherwise $v$ would have at least $(k-1)m$ out-neighbors. Now we color $v$ with $c$, and claim that the subgraph $H$ induced by all vertices of color $c$ is $m$-degenerate. To see this, let $H'$ be an induced subgraph of $H$. If $v \in V(H')$, then notice that $v$ has at most $m-1$ out-neighbors in $H'$. Otherwise, note that $H'$ is a subset of a color class in a $(k-1, m)$-degenerate coloring of $D - v$, meaning that there is some vertex in $H'$ of in-degree or out-degree less than $m$. 

By reversing all edges in $D$, we symmetrically obtain that $d^-(v) \geq (k-1)m$.
\end{proof}

A digraph is {\it weakly connected} if the underlying undirected graph is connected. 
Our next lemma states that we only need to consider the weakly connected components of a digraph to find its $m$-degenerate chromatic number.
\begin{lemma}
\label{lem:trivial}
If $D$ is a digraph and $D_1, \ldots, D_l$ are its weakly connected components for some positive integer $l$, then for any $m \geq 1$, $\chi_m(D) = \max_{1 \leq i \leq l} \chi_m(D_i)$.
\end{lemma}
\begin{proof}
Let $k = \max_{1 \leq i \leq l} \chi_m(D_i)$. We can find a $(k, m)$-degenerate coloring of each $D_i$, for $1 \leq i \leq l$, and the resulting composite coloring is a $(k, m)$-degenerate coloring of $D$ since there is no edge between any two weakly connected components of $D$.
\end{proof}
By Lemma \ref{lem:trivial}, a digraph $D$ which is $(k, m)$-critical is also weakly connected.
Our next result, Theorem \ref{thm:kmcritical} states that the in-degree and out-degree of every vertex cannot be too small in a $(k, m)$-critical oriented graph with $k > 2$.

\begin{theorem}
\label{thm:kmcritical}
Suppose that $m \geq 1$ and $D$ is a $(k, m)$-critical oriented graph on $n$ vertices in which each vertex $v$ satisfies $d^+(v) = d^-(v) = (k-1)m$. Then $k \leq 2$.
\end{theorem}
Theorem \ref{thm:kmcritical} is of particular interest since it generalizes the following theorem of Mohar, who proved the case $m = 1$, which is a statement about digraph colorings.

\begin{theorem}[Mohar \cite{mohar_eigen_2009}]
\label{thm:moharkcritical}
If $D$ is a $(k, 1)$-critical oriented graph in which each vertex $v$ satisfies $d^+(v) = d^-(v) = k-1$, then $k \leq 2$.
\end{theorem}

In the following lemma, we prove that if an oriented graph $D$ is $(k,m)$-critical, then there is a set of $m+1$ in-neighbors or out-neighbors of a vertex which, when removed, does not break weakly connectedness of $D$. In the proof of Theorem \ref{thm:moharkcritical} \cite{mohar_eigen_2009}, this fact for the case $m = 1$ was assumed to be true for all digraphs $D$, including those with digons, but not explicitly stated. 

\begin{lemma}
\label{lem:2separator}
Suppose $m \geq 1$, $k \geq 3$, and $D$ is a $(k, m)$-critical oriented graph on $n$ vertices in which each vertex $v$ satisfies $d^+(v) = d^-(v) = (k-1)m$. Then there exist vertices $u_1, u_2, \ldots, u_{m+1}, w \in V(D)$ such that $u_1, \ldots, u_{m+1}$ are all out-neighbors or all in-neighbors of $w$ and the digraph induced by $V(D) - \{u_1, \ldots, u_{m+1}\}$ is weakly connected.
\end{lemma}

\begin{proof}
Suppose the lemma is false for some $(k,m)$-critical oriented graph $D$; by Lemma \ref{lem:trivial}, $D$ is weakly connected.  Let $X$ be a set of $m+1$ in-neighbors or out-neighbors of some vertex $w_0$ that maximizes the lexicographic size of the weakly connected components of $D - X$, when listed from largest to smallest; the size of a weakly connected component is its number of vertices. If $D - X$ has a single weakly connected component, then taking $w = w_0$ and $u_1, \ldots, u_{m+1}$ as the $m+1$ vertices in $X$ satisfies the statement of the lemma. Otherwise, let $C$ be the smallest weakly connected component of $D - X$, and $C'$ be another weakly connected component. 

Some $x \in X$ must be adjacent to $C'$, or else $C'$ would be its own weakly connected component in $D$. Next pick any vertex $c \in C$, and notice that all neighbors of $c$ are contained in $C \cup X$.  Assume $x$ is not an in-neighbor of $c$; the case $x$ is not an out-neighbor of $c$ is symmetric, by reversing all edges. Let $Y$ be an arbitrary set of $m+1$ in-neighbors of $c$. Since $Y \subseteq C \cup X$, each component of $D - X$ that is not $C$ is also weakly connected in $D - Y$. By maximality of the lexicographic sizes of the components of $D - X$, and since $C$ is a smallest weakly connected component of $D-X$, each weakly connected component of $D - X$ that is not $C$ must not contain any additional vertices in $D - Y$. But $x \in D-Y$ is in the same weakly connected component as $C'$, which is a contradiction.
\end{proof}

We now prove Theorem $\ref{thm:kmcritical}$, using Lemma \ref{lem:2separator}.

\begin{proof}[Proof of Theorem \ref{thm:kmcritical}] We assume that $k \geq 3$ for the purpose of contradiction and create a linear ordering of the vertices of $D$, as follows. Pick a vertex $w \in D$ and choose a set $U = \{u_1, u_2, \ldots, u_{m+1}\}$ of $m+1$ vertices in the in-neighborhood or out-neighborhood of $w$ so that the digraph  $D' := D - U$ is weakly connected. This construction is possible by Lemma \ref{lem:2separator}. We will now form a {\it degeneracy ordering} of the vertices of $D - w$, which is an ordering such that each vertex has strictly fewer than $(k-1)m$ in-neighbors or $(k-1)m$ out-neighbors before itself in the ordering. 
This degeneracy ordering also has the property that the first $m+1$ vertices are $u_1, \ldots, u_{m+1}$. We next order the remaining vertices in reverse order, starting with $u_{n-1}$ and ending with $u_{m+2}$.

Since $w$ has $(k-1)m$ in-neighbors and $(k-1)m$ out-neighbors, there is some $u_{n-1} \in D'$ such that $w$ is an out-neighbor or in-neighbor of $u_{n-1}$. Thus, $u_{n-1}$ has strictly fewer than $(k-1)m-1$ out-neighbors or in-neighbors in $D - w$. Now continue in a similar manner; namely, for each $i$, $n-2 \geq i \geq m+2$, since $D'$ is weakly connected there is some vertex $u_i \in V(D') \backslash \{u_{i+1}, u_{i+2}, \ldots, u_{n-1}, w\}$ which has an in-neighbor or out-neighbor in the set $\{u_{i+1}, u_{i+2}, \ldots, u_{n-1}, w\}$. Hence $u_i$ has out-degree or in-degree less than $(k-1)m$ in the digraph $D - \{u_{i+1}, u_{i+2}, \ldots, u_{n-1}, w\}$. Since, for $1 \leq i \leq m+1$, $u_i$ clearly has fewer than $(k-1)m$ in-neighbors or out-neighbors among the vertices $u_1, \ldots, u_{i-1}$, the construction of the degeneracy ordering is complete.


We now construct a $(k-1,m)$-degenerate coloring of $D$ as follows, which will contradict the fact that $\chi_m(D) = k$. We first give $u_1, u_2, \ldots, u_{m+1}$ the same color. 
Then for $m+2 \leq i \leq n-1$, we assign $u_i$ a color, as follows: in the subgraph of $D$ induced by $\{u_1, \ldots, u_i\}$, $u_i$ has in-degree or out-degree less than $(k-1)m$. Therefore, in the existing coloring of $\{u_1, \ldots, u_{i-1}\}$, one of the $k-1$ color classes contains fewer than $m$ in-neighbors or out-neighbors of $u_i$, and we give $u_i$ this color. 
Finally, to color $u_n$, note that $u_n$ has $m+1$ in-neighbors or out-neighbors, namely $u_1, \ldots, u_{m+1}$, of the same color, but exactly $(k-1)m$ in-neighbors and out-neighbors in total, so we can find a color present in fewer than $m$ in-neighbors or out-neighbors of $u_n$; we color $u_n$ this color. Note that the resulting coloring has the property that each vertex $u_i$ has fewer than $m$ in-neighbors or out-neighbors of the same color as $u_i$ which belong to $\{u_1, \ldots, u_{i-1}\}$.

We claim that each color class $C$ is $m$-degenerate. To show this, for any color class $C$ and  subset $S$ of the vertices colored $C$, pick $v \in S$ by $v = w$ if $w \in S$, and otherwise $v = u_i \in S$ such that $i$ is as large as possible. Then since $v$ is the vertex in $S$ that was colored last, $v$ has at most $m-1$ in-neighbors or out-neighbors in $S$, completing the proof.
\end{proof}


Our next lemma uses Lemma \ref{lem:critvertex} to extend Theorem \ref{thm:kmcritical} to digraphs that are not $(k, m)$-critical. Intuitively, this is possible because $(k, m)$-critical digraphs are the worst case for finding an $m$-degenerate coloring with few colors.

\begin{lemma}
\label{lem:brooksdeg}
Suppose that $m \geq 1$ and $\chi_m(D) = k + 1$, for some integer $k \geq 2$ and oriented graph $D$. Then ${\Delta}(D)/2 >   km$.
\end{lemma}
\begin{proof}
Fix $m \geq 1$.
Suppose for the purpose of contradiction that for some $k \geq 2$, there is an oriented graph $D$ with as few vertices as possible, such that $\chi_m(D) = k+1$ and ${\Delta}(D) \leq 2km$.  Notice that if $D$ were not $(k+1, m)$-critical, we could remove some vertex $v$ to form $D' = D - v$, and we would have $\chi_m(D') = k+1$ and ${\Delta}(D') \leq {\Delta}(D) \leq 2km$. This contradicts the fact that $D$ has as few vertices as possible such that ${\Delta}(D) \leq 2km$ holds. Hence $D$ is $(k+1, m)$-critical. 

By Lemma \ref{lem:critvertex}, for each $v \in V(D)$, we have that $d^+(v) \geq km$ and $d^-(v) \geq km$. In order to have ${\Delta}(D) \leq 2km$, we must have $d^+(v)  = d^-(v) = km$ for all $v \in V(D)$. But then by Theorem \ref{thm:kmcritical}, we have that $k+1 \leq 2$, contradicting the fact  that $k \geq 2$.
\end{proof}

The following corollary  follows from Lemma \ref{lem:brooksdeg}. It is a directed analogue of a theorem of Borodin \cite{borodin_decomposition_1976}, which states that the $m$-degenerate chromatic number of an undirected graph $G$ is at most $\left\lceil \frac{\Delta(G)}{m}\right\rceil$ as long as $\Delta(G) \geq \max\{3, 2m, \omega(G)\}$, where $\omega(G)$ denotes the clique number of $G$.

\begin{corollary}
\label{cor:brookscor}
If $D$ is an oriented graph such that 
${\Delta}(D) > 2m$, then $\chi_m(D) \leq \left \lceil \frac{{\Delta}(D)}{2m}\right\rceil$.
\end{corollary}
\begin{proof}
Let $k = \left\lceil \frac{{\Delta}(D)}{2m} \right\rceil$. Then $k \geq 2$. If $\chi_m(D) \geq k + 1$, then by Lemma \ref{lem:brooksdeg}, ${\Delta}(D) > 2km$, meaning that $\frac{{\Delta}(D)}{2m} > k$, a contradiction since $k = \left\lceil \frac{{\Delta}(D)}{2m} \right\rceil$.
\end{proof}



To prove Theorem \ref{thm:fracdelta}, we also use a well-known decomposition theorem of Lov\' asz \cite{lovasz_on_1966}. 

\begin{theorem}[Lov\' asz \cite{lovasz_on_1966}]
\label{thm:lovasz}
For an undirected graph $G$, suppose that for some $s \geq 1$ and non-negative integers $\Delta_1, \ldots, \Delta_s$, we have $\Delta(G) = -1 + \sum_{i = 1}^s (\Delta_i + 1)$. Then there is a partitioning of $V(G)$ into $s$ sets $V_i$ which induce subgraphs $G_i$ $(1 \leq i \leq s)$, such that $\Delta(G_i) \leq \Delta_i$.
\end{theorem}

%
%
%

We now prove Theorem \ref{thm:fracdelta}, using Lemma \ref{lem:brooksdeg} and Theorem \ref{thm:lovasz}  to find an upper bound on $\chi_m(D)$.

\begin{proof}[Proof of Theorem \ref{thm:fracdelta}]
Notice that if $\Delta(D) < 4m$, then we have
\begin{equation}
\chi_m(D) \leq \left\lfloor \frac{\Delta(D)}{2m} \right\rfloor + 1 = \left\lfloor \frac{\Delta(D) - \left\lfloor \frac{\Delta(D) + 1}{4m + 1} \right\rfloor}{2m} \right\rfloor + 1\nonumber,
\end{equation}
so we assume from here on that $\Delta(D) \geq 4m$.

To find an upper bound on $\chi_m(D)$, we will use Theorem \ref{thm:lovasz} to show that the vertices of  $D$ can be partitioned into several subsets, each inducing a subgraph $D_i \subset D$, such that ${\Delta}(D_i)$ is small (at most $4m$). We will then apply Lemma \ref{lem:brooksdeg} to show that $\chi_m(D_i)$ is also small, meaning that we may color each $D_i$ with its own set of $\chi_m(D_i)$ colors to obtain a composite coloring which is also clearly $m$-degenerate; this argument gives the bound \mbox{$\chi_m(D) \leq \sum_{i} \chi_m(D_i)$}. 

We first set
\begin{equation}
t = \left\lfloor \frac{{\Delta}(D) + 1}{4m + 1} \right\rfloor, \quad r = {{\Delta}(D) + 1 - t(4m + 1)}. \nonumber
\end{equation} 
 It is clear that $r \geq 0$. Notice that ${\Delta}(D) = -1 + \left(\sum_{i = 1}^t 4m+1\right) + r$, meaning that, if $r \geq 1$, we may apply Theorem \ref{thm:lovasz} with $s = t+1$, $\Delta_i = 4m$ for $1 \leq i \leq t$, and $\Delta_{t+1} = r-1$. If $r = 0$, then we apply Theorem \ref{thm:lovasz} with $s = t$ and $\Delta_i = 4m$ for $1 \leq i \leq t$. Hence, if $r \geq 1$, the vertices of $D$ can be partitioned into $t + 1$ sets inducing subgraphs $D_1, \ldots, D_{t+1}$ (if $r = 0$, then $t$ sets inducing subgraphs $D_1, \ldots, D_t$), which satisfy:
\begin{equation}
\label{eq:deltadi}
{\Delta}(D_i) \leq
\begin{cases}
4m \quad\mbox{ if }\quad 1 \leq i \leq t\\
r - {1} \quad \mbox{ if } \quad r \geq 1\mbox{ and } i = t+1.
\end{cases}
\end{equation}

By Lemma \ref{lem:brooksdeg}, for $1 \leq i \leq t$, if $\chi_m(D_i) \geq 3$, then we would have that ${\Delta}(D_i) > 4m$, a contradiction to (\ref{eq:deltadi}). So $\chi_m(D_i) \leq 2$ for $1 \leq i \leq t$. If $r \geq 1$, we recall also the trivial bound $\chi_m(D_{t+1}) \leq \left\lfloor \frac{{\Delta}(D_{t+1})}{2m}\right \rfloor + 1$. Therefore we have, by (\ref{eq:deltadi}),
\begin{equation}
\label{eq:chimdi}
\chi_m(D_i) \leq \begin{cases}
2 \quad \mbox{ if } \quad 1 \leq i \leq t\\
1 + \left\lfloor \frac{r-1}{2m}\right\rfloor \quad \mbox{ if } \quad r \geq 1 \mbox{ and } i = t + 1 .
\end{cases}
\end{equation}
Notice that if $r = 0$, then $1 + \lfloor (r-1)/(2m) \rfloor = 0$.
Combining the individual $m$-degenerate colorings of $D_i$ into an $m$-degenerate coloring of $D$, we obtain
\begin{eqnarray}
\chi_m(D)& \leq& \sum_{i = 1}^{t+1} \chi_m(D_i)\nonumber\\
&\leq& 2 t + \lfloor (r - 1)/(2m) \rfloor + 1 \nonumber\\
&=& 2 \cdot \left\lfloor \frac{{\Delta}(D) + 1}{4m + 1} \right\rfloor + \left \lfloor \frac{{\Delta}(D) - (4m + 1) \cdot \left\lfloor \frac{{\Delta(D)} + 1}{4m + 1} \right\rfloor}{2m}\right\rfloor + 1\nonumber\\
&=& \left\lfloor \frac{{\Delta}(D)}{2m} -\left(\frac{4m + 1}{2m} -2 \right)  \left\lfloor \frac{{\Delta}(D) + 1}{4m + 1} \right\rfloor\right\rfloor + 1\nonumber\\
&=& \left\lfloor \frac{{\Delta}(D)  -  \left\lfloor \frac{{\Delta}(D) + 1}{4m + 1} \right\rfloor}{2m}\right\rfloor + 1\nonumber,
\end{eqnarray}
where we have used (\ref{eq:chimdi}) in the second inequality and the definition of $t$ and $r$ in the subsequent equality.
\end{proof}
{\bf Remark.} If, in applying Theorem \ref{thm:lovasz}, we had tried to set each $\Delta_i$ to an integer smaller than $4m$ in an attempt to further decrease $\chi_m(D_i)$, it would not necessarily be true that $\chi_m(D_i) \leq 1$; hence our final bound would not improve. However, for a certain class of digraphs in the case $m = 1$, this strategy will work, as we investigate in Section \ref{sec:betbounds}.

Theorem \ref{thm:fracdelta} and its Corollary \ref{cor:acyclic} give bounds on $\chi_m(D)$ in terms of $\Delta(D) = 2\bar{\Delta}(D) \geq 2\tilde{\Delta}(D)$. In order to directly compare our bond on $\chi_A(D) = \chi_1(D)$ to the bound of $\chi_A(D) \leq (1 - e^{-13}) \cdot \tilde{\Delta}(D)$ by Harutyunyan and Mohar, we need the following theorem.

\begin{theorem}
\label{thm:convertgeo}
Fix $m \geq 1$.
Suppose $\mathcal{D}$ is a set of digraphs which is closed under taking induced subgraphs, and that $a, b, c \in \mathbb{R^{\geq \mbox{\footnotesize$0$}}}$ with $b \geq 1$ and $0<a \leq 1/2$. Suppose further that for each $D \in \mathcal{D}$ with $\Delta(D) \geq c$, we have 
\begin{equation}
\label{eq:arithmeangiv}
\chi_m(D) \leq \left\lfloor a \cdot \frac{\Delta(D)}{m} + b\right\rfloor.
\end{equation}
Then for each $D \in \mathcal{D}$ with $\tilde{\Delta}(D) \geq c \cdot 1/2 \cdot \sqrt{(1-a)/a}$, we have $$\chi_m(D) \leq \left\lfloor \sqrt{\frac{a}{1-a}}\cdot \frac{\tilde{\Delta}(D)}{m} + b\right\rfloor.$$

\end{theorem}

\begin{proof}
Set $r = \sqrt{a/(1-a)}$, and pick any $D \in \mathcal{D}$ with $\tilde{\Delta}(D) \geq \frac{c}{2r}$. Our strategy is to use (\ref{eq:arithmeangiv}) on a subgraph $D'$ of $D$ that is induced by the vertices of $D$ which have large out-degree and in-degree. Such vertices cannot have too large a total degree, or else the geometric mean of the in-degree and out-degree of such vertices would be greater than $\tilde\Delta(D)$. This argument implies the existence of a coloring of $D'$ with a relatively small number of colors, and then we can greedily extend the coloring to the remaining vertices of $D$ by using the fact that either the in-degree of out-degree of each vertex in $V(D) \backslash V(D')$ is small.

Let $D'$ be the subgraph induced by the set of all $v \in V(D)$ such that $d^+(v) > r \tilde \Delta(D)$ and $d^-(v) > r \tilde \Delta(D)$. By definition, $d^+(v)d^-(v) \leq \tilde \Delta(D)^2$ for each $v \in D$. Since the sum $d^+(v) + d^-(v)$ is maximized subject to the constraint $d^+(v)d^-(v) \leq \tilde \Delta(D)^2$ when $d^+(v)$ and $d^-(v)$ are as far apart as possible, we have, for each $v \in V(D')$, 
\begin{equation}
d^+(v) + d^-(v) \leq r \tilde \Delta(D) + \frac{\tilde \Delta(D)^2}{r \tilde \Delta(D)} = \left( r + \frac{1}{r} \right) \tilde \Delta(D)\nonumber.
\end{equation}
Hence ${\Delta}(D') \leq \left(r + \frac{1}{r} \right)\tilde \Delta(D)$. Since $\mathcal{D}$ is closed under taking induced subgraphs, $D' \in \mathcal{D}$. Moreover, we have that $\Delta(D') > 2r \tilde \Delta(D) \geq c$,
so by (\ref{eq:arithmeangiv}),
\begin{equation}
\label{eq:coloringdprime}
\chi_m(D') \leq \left\lfloor a \cdot \frac{{\Delta}(D')}{m} + b \right\rfloor \leq \left\lfloor \frac{a\left(r+\frac{1}{r}\right) \tilde \Delta(D)}{m}+b \right\rfloor.
\end{equation}
Next, given an $m$-degenerate coloring of $D'$ with some number $k$ of colors, we claim that we can extend this coloring to an $m$-degenerate coloring of $D$ with $\max\{k, \lfloor\lfloor r \tilde \Delta(D) \rfloor /m\rfloor + 1\}$ colors. Suppose we are given such a coloring of $D'$; we then color the vertices of $V(D) \backslash V(D')$ in any order, noting, for each vertex $v \in V(D) \backslash V(D')$, that either $d^+(v) \leq \lfloor r \tilde \Delta(D) \rfloor$ or $d^-(v) \leq \lfloor r \tilde \Delta(D) \rfloor$. Therefore, there are at most $\lfloor \lfloor r \tilde \Delta(D) \rfloor/ m \rfloor$ colors which are present in at least $m$ in-neighbors (or out-neighbors) of $v$. Thus, among either the out-neighbors or in-neighbors of $v$, there is some color not present, and we color $v$ that color. 

In the resulting coloring of $D$, let $H$ be the subgraph induced by the vertices of any given color. If $H'$ were some induced subgraph of $H$ with no vertex of in-degree or out-degree strictly less than $m$, then $H'$ must contain some vertex in $V(D) \backslash V(D')$. Pick the $v \in (V(D) \backslash V(D')) \cap V(H')$ which was colored last of all vertices of $H'$, and note that $v$ has at most $m-1$ in-neighbors or out-neighbors of the same color. Hence $H$ is weakly $m$-degenerate, and the resulting coloring of $D$ is indeed an $m$-degenerate coloring. 
Since (\ref{eq:coloringdprime}) gives an upper bound on the $m$-degenerate chromatic number of $D'$, the preceding argument shows that
\begin{eqnarray}
\chi_m(D) &\leq&  \max\left\{\lfloor \lfloor r \tilde \Delta(D) \rfloor/m+ 1 \rfloor, \left\lfloor  \frac{a\left(r+\frac{1}{r}\right) \tilde \Delta(D)}{m}+b \right\rfloor\right\} \nonumber\\
&\leq&  \left\lfloor \frac{\tilde \Delta(D)}{m} \cdot \max\left\{r, a \cdot \left( r + \frac{1}{r}\right)\right\} + b \right\rfloor\nonumber\\
&=&   \left\lfloor\frac{r \tilde \Delta(D)}{m} + b\right\rfloor\nonumber,
\end{eqnarray}
where we have used in the second inequality that $b \geq 1$ and in the following equality that 
\begin{eqnarray}
a\left(r+ \frac{1}{r}\right) = ra\left(1 + \frac{1}{r^2}\right) = ra\left(1 + \frac{1-a}{a}\right) =r\nonumber.
\end{eqnarray}
\end{proof}

It follows from Theorems \ref{thm:fracdelta} and \ref{thm:convertgeo} that we can obtain, for all $m \geq 1$, an upper bound on $\chi_m(D)$ in terms of $\tilde{\Delta}(D)$.
\begin{theorem}
\label{cor:mdegtilde}
Let $m$ be a positive integer. Then for any oriented graph $D$, we have
\begin{equation}
\chi_m(D) \leq \left\lfloor \sqrt{\frac{2m}{2m+1}} \cdot \frac{\tilde\Delta(D)}{m} + \frac{2}{4m+1} \right\rfloor + 1\nonumber.
\end{equation}
\end{theorem}
\begin{proof}
By Theorem \ref{thm:fracdelta}, we have that for any oriented graph $D$,
\begin{eqnarray}
\chi_m(D) &\leq&  \left\lfloor \frac{{\Delta}(D)  -  \left\lfloor \frac{{\Delta}(D) + 1}{4m + 1} \right\rfloor}{2m}\right\rfloor + 1\nonumber\\
&\leq& \left\lfloor \frac{2m}{4m+1} \cdot \frac{\Delta(D)}{m} + \frac{4m+3}{4m+1} \right\rfloor\nonumber.
\end{eqnarray}
The result now immediately follows from Theorem \ref{thm:convertgeo} with $a = (2m)/(4m+1)$, $b = (4m+3)/(4m+1)$, and $c = 0$. 
\end{proof}

By taking $m = 1$ in Theorem \ref{cor:mdegtilde}, we obtain Corollary \ref{cor:improvedbound}. 
Harutyunyan and Mohar \cite{strengthened_harut_2011} posed the question of determining the smallest integer $\Delta_0$ such that every oriented graph $D$ with $\lceil \tilde{\Delta}(D) \rceil = \Delta_0$ satisfies $\chi_A(D) \leq \Delta_0 - 1$. They showed that $\Delta_0$ exists and is at most some integer which is approximately equal to $10^{10}$, and they also proved that $\Delta_0 \geq 4$. From Corollary \ref{cor:improvedbound}, we have that $\Delta_0 \leq 8$:

\begin{corollary}
Any oriented graph $D$ with $\lceil \tilde{\Delta}(D) \rceil = 8$ has $\chi_A(D) \leq 7$.
\end{corollary}
\begin{proof}
By Corollary \ref{cor:improvedbound}, we have that if $\lceil \tilde\Delta(D) \rceil = 8$, then
\begin{equation}
\chi_A(D) \leq \left\lfloor\sqrt{2/3} \cdot \tilde{\Delta}(D) + \frac{7}{5}\right\rfloor \leq \left\lfloor\sqrt{2/3} \cdot 8 + \frac{7}{5}\right\rfloor = 7\nonumber.
\end{equation}
\end{proof}

\section{Strengthened bounds for $\chi_A(D)$ with forbidden subgraphs}
\label{sec:betbounds}
In this section, we show that the bound in Theorem \ref{thm:lovasz} can be improved for any digraph which does not contain any of the graphs shown in Figure \ref{fig:avoid} as an induced subgraph. This leads to an improvement to the bound in Theorem \ref{thm:fracdelta} in the case $m = 1$ for such digraphs.  We first introduce some notation. Given disjoint sets of vertices $V_1$ and $V_2$ which belong to a digraph $D$, we let $E(V_1, V_2) = E(V_2, V_1)$ be the set of all edges which connect a vertex in $V_1$ to a vertex in $V_2$. Moreover, we define $e(V_1, V_2) = |E(V_1, V_2)|$. Given a vertex $u$, we define $d_{V_1}^+(u)$ as the number of out-neighbors of $u$ in $V_1$, $d_{V_1}^-(u)$ as the number of in-neighbors of $u$ in $V_1$, and ${d}_{V_1}(u) = {d_{V_1}^+(u) + d_{V_1}^-(u)}$. Finally, we let $F_1, F_2, G_1$, and $G_2$ be the 4-vertex digraphs shown in Figure \ref{fig:avoid}. 

\begin{figure}

\begin{minipage}[b]{0.22\linewidth}
\begin{center}
$F_1$\\
\begin{tikzpicture}

\tikzset{vertex/.style = {shape=circle,draw,minimum size=1.5em}}
\tikzset{edge/.style = {->,> = latex'}}
\node[shape=circle,minimum size=1em,fill=gray] (w) at  (0,0) {};
\node[shape=circle,minimum size=1em,fill=gray] (y) at  (0,2) {};
\node[shape=circle,minimum size=1em,fill=gray] (q) at  (-2,0) {};
\node[shape=circle,minimum size=1em,fill=gray] (r) at  (-2,2) {};

\draw[edge] (q) to (r);
\draw[edge] (q) to (w);
\draw[edge] (r) to (y);
\draw[edge] (w) to (y);
\end{tikzpicture}
\end{center}

\end{minipage}
\hspace{0.01cm}
\begin{minipage}[b]{0.22\linewidth}
\begin{center}
$F_2$\\
\begin{tikzpicture}

\tikzset{vertex/.style = {shape=circle,draw,minimum size=1.5em}}
\tikzset{edge/.style = {->,> = latex'}}
\node[shape=circle,minimum size=1em,fill=gray] (w) at  (0,0) {};
\node[shape=circle,minimum size=1em,fill=gray] (y) at  (0,2) {};
\node[shape=circle,minimum size=1em,fill=gray] (q) at  (-2,0) {};
\node[shape=circle,minimum size=1em,fill=gray] (r) at  (-2,2) {};

\draw[edge] (q) to (r);
\draw[edge] (q) to (w);
\draw[edge] (r) to (y);
\draw[edge] (w) to (y);
\draw[edge] (w) to (r);
\end{tikzpicture}
\end{center}
\end{minipage}
\hspace{.01cm}
\begin{minipage}[b]{0.25\linewidth}
\begin{center}
$G_1$\\
\begin{tikzpicture}

\tikzset{vertex/.style = {shape=circle,draw,minimum size=1.5em}}
\tikzset{edge/.style = {->,> = latex'}}
\node[shape=circle,minimum size=1em,fill=gray] (w) at  (0,0) {};
\node[shape=circle,minimum size=1em,fill=gray] (y) at  (0,2) {};
\node[shape=circle,minimum size=1em,fill=gray] (q) at  (-2,0) {};
\node[shape=circle,minimum size=1em,fill=gray] (r) at  (-2,2) {};

\draw[edge] (q) to (r);
\draw[edge] (q) to (w);
\draw[edge] (r) to (y);
\draw[edge] (w) to (y);
\draw[edge] (y) to (q);
\end{tikzpicture}
\end{center}
\end{minipage}
\hspace{.01cm}
\begin{minipage}[b]{0.22\linewidth}
\begin{center}
$G_2$\\
\begin{tikzpicture}

\tikzset{vertex/.style = {shape=circle,draw,minimum size=1.5em}}
\tikzset{edge/.style = {->,> = latex'}}
\node[shape=circle,minimum size=1em,fill=gray] (w) at  (0,0) {};
\node[shape=circle,minimum size=1em,fill=gray] (y) at  (0,2) {};
\node[shape=circle,minimum size=1em,fill=gray] (q) at  (-2,0) {};
\node[shape=circle,minimum size=1em,fill=gray] (r) at  (-2,2) {};

\draw[edge] (q) to (r);
\draw[edge] (q) to (w);
\draw[edge] (r) to (y);
\draw[edge] (w) to (y);
\draw[edge] (w) to (r);
\draw[edge] (y) to (q);
\end{tikzpicture}
\end{center}
\end{minipage}
\caption{The digraphs $F_1, F_2, G_1$, and $G_2$.}
\label{fig:avoid}
\end{figure}
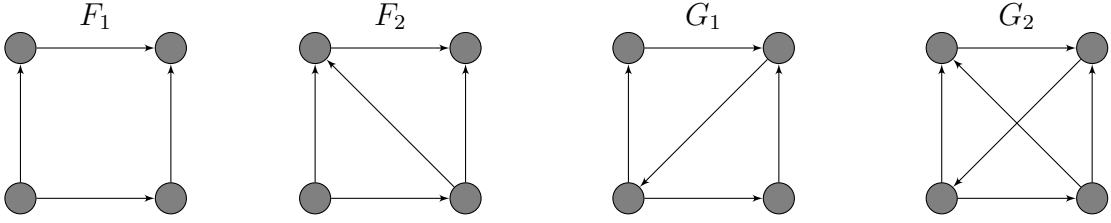

The following lemma shows that if $D$ contains none of $F_1, F_2, G_1$, or $G_2$ as induced subgraphs, then we can partition $V(D)$ into some number $s$ of sets which induce subgraphs $D_1, \ldots, D_s$, so that we may find an upper bound on $\chi_A(D_i)$ for each $i$. Notice that if the shortest directed cycle in $D$ is of length at least 4, then the constraint is relaxed to the condition that $D$ contains neither $F_1$ nor $F_2$ as an induced subgraph. The proof follows that of an undirected analogue proved by Catlin \cite{catlin_another_1978}. Both our proof and Catlin's proceed by picking a partition which maximizes the same function of the partition and then using the forbidden subgraphs to show that such a partition gives the desired bound on the chromatic number of each of $D_1, \ldots, D_s$. In Catlin's proof, the forbidden induced subgraphs were all 4-vertex graphs with a 4-cycle.

\begin{lemma}
\label{lem:modlov}
Suppose we are given a digraph $D$, and non-negative integers $s,{\Delta}_1, \ldots, {\Delta}_s$. Suppose further that ${\Delta}(D) = -2 + \sum_{i = 1}^s (\Delta_i + 1)$ and $D$ does not contain any of $F_1, F_2, G_1, G_2$ as an induced subgraph. Then there is a partitioning  of $V(D)$ into $s$ subsets $V_i$ which induce subgraphs $D_i$ ($1 \leq i \leq s$), such that $\Delta(D_i) \leq \Delta_i$ for all $i$ and $\chi_A(D_i) \leq \lceil{\Delta}_i/2\rceil$ for all $i$ with $\Delta_i > 0$.
\end{lemma}
\begin{proof}
Given a partition of $V(D)$ into $s$ subsets $V_1, \ldots, V_s$, define
\begin{equation}
\label{eq:defoff}
f(V_1, \ldots, V_s) = {\Delta}_1 |V_1| + {\Delta}_2 |V_2|+ \cdots + {\Delta}_s |V_s| + \sum_{1 \leq i < j \leq s} e(V_i, V_j).
\end{equation}
We let $A$ be the set of all $i$ such that ${\Delta}_i = 2$. Notice that if $i \in A$, then any directed cycle in $V_i$ must be both induced and a component of the subgraph induced by $V_i$. Following notation of Catlin \cite{catlin_another_1978}, we call any such directed cycle a {\it Brooks cycle}.
Now, choose a partition of the vertices of $D$ into subsets $V_1, \ldots, V_s$ so that, (i), $f(V_1, \ldots, V_s)$ is maximized, and (ii), the number of Brooks cycles is minimized, subject to (i). 

We first claim that any partition which maximizes $f$ has ${\Delta}(D_i) \leq {\Delta}_i$ for all $i$. Let $V_1, \ldots, V_s$ be a partition of $V(D)$ which maximizes $f$, and notice that for any $u \in V_1$, and for $2 \leq j \leq s$, moving $u$ from $V_1$ to $V_j$ must not increase the value of $f$:
\begin{equation}
\label{eq:from1toj}
f(V_1, \ldots, V_s) - f(V_1 -u, V_2, \ldots, V_j + u, \ldots, V_s) \geq 0.
\end{equation}
Using the definition of $f$ in (\ref{eq:defoff}), we obtain from (\ref{eq:from1toj}) that
\begin{eqnarray}
0&\leq&f(V_1, \ldots, V_s) - f(V_1 -u, V_2, \ldots, V_j + u \ldots, V_s)\nonumber\\
\label{eq:finalmove}
 &=& {e(V_1, V_j) - e(V_1 - u, V_j + u) } + {\Delta}_1 - {\Delta}_j\\
 &=& {d}_{V_j}(u) - {d}_{V_1}(u) + {\Delta}_1 - {\Delta}_j.\nonumber
\end{eqnarray}
Therefore, for $1 \leq j \leq s$, by (\ref{eq:finalmove}),
\begin{equation}
\label{eq:toavg}
{d}_{V_1}(u) \leq {\Delta}_1 - {\Delta}_j + {d}_{V_j}(u).
\end{equation}
By averaging (\ref{eq:toavg}) over all choices of $j$, $1 \leq j \leq s$, we obtain
\begin{equation}
{d}_{V_1}(u) \leq {\Delta}_1 - \frac{\sum_{j = 1}^s {\Delta}_j}{s} + \frac{{d}(u)}{s} = {\Delta}_1 - \frac{{\Delta}(D) - (s-2)}{s} + \frac{{d}(u)}{s}\nonumber.
\end{equation}
But ${d}(u) \leq {\Delta}(D)$, so the above implies that
\begin{equation}
{d}_{V_1}(u) \leq {\Delta}_1 + \frac{s-2}{s}\nonumber.
\end{equation}
Since ${\Delta}_1$ and ${d}_{V_1}(u)$ are integers, we must have that ${d}_{V_1}(u) \leq {\Delta}_1$. Since $u$ can be any vertex in $V_1$, we have ${\Delta}(D_1) \leq {\Delta}_1$. Moreover, we can repeat the above process with $V_1$ replaced by $V_i$, for $2 \leq i \leq s$, so ${\Delta}(D_i) \leq {\Delta}_i$.

For all $i$ with $\Delta_i \geq 3$, by Corollary \ref{cor:brookscor}, we have that $\chi_A(D_i) \leq \lceil{\Delta}(D_i)/2 \rceil$. For all $i$ with $\Delta_i = 1$, we clearly have that $\chi_A(D_i) = 1 = \lceil \Delta_i/2 \rceil$. It remains to consider those $i$ for which $\Delta_i = 2$, meaning $i \in A$. We claim that since the partition $V_1, \ldots, V_{s}$ minimizes the total number of Brooks cycles subject to the fact that $f(V_1, \ldots, V_s)$ is maximized,  the total number of Brooks cycles is 0, which implies that $D_i$ is acyclic for $i \in A$. For the purpose of contradiction, suppose there is some Brooks cycle $C_0 \subset V_i$, and let $v_0 \in V(C_0)$. If ${d}_{V_j}(v_0) \geq {\Delta}_j + 1$ for each $j \neq i$, then since ${d}_{V_i}(v_0) = {\Delta}_i = 2$,
\begin{equation}
{d}(v_0) \geq {\Delta}_i + \sum_{j \neq i} \left({\Delta}_j + 1\right) = {\Delta}(D) + 1  > {\Delta}(D)\nonumber,
\end{equation}
which is impossible. Hence there is some $j \neq i$ so that ${d}_{V_j}(v_0) \leq {\Delta}_j$. Thus, noting that ${d}_{V_i}(u) = {\Delta}_i$, and since
\begin{equation}
f(V_1, \ldots, V_s) - f(V_1, \ldots, V_i - u, \ldots, V_j + u, \ldots, V_s) = {d}_{V_j}(u) - {d}_{V_i}(u) + {\Delta}_i - {\Delta}_j\nonumber,
\end{equation} moving $v_0$ from $V_i$ to $V_j$ does not decrease $f(V_1, \ldots, V_s)$. Moreover, since moving $v_0$ removes the Brooks cycle $C_0$ from $V_i$, it must create a Brooks cycle $C_1$ in $V_j$. Therefore, by our definition of Brooks cycle, ${\Delta}_j = 2$, so $j \in A$. We then pick $v_1 \neq v_0$ such that $v_1 \in V(C_1)$, and repeat the process, creating an infinite sequence of Brooks cycles, $C_1, C_2, C_3, \ldots$. In particular, in the $t$th iteration of this process ($t \geq 1$), we create the Brooks cycle $C_t$. 

Note that in this sequence of Brooks cycles, any two adjacent cycles $C_t$ and $C_{t+1}$ must share a vertex, namely the vertex that is moved from $C_t$ to create $C_{t+1}$. Since the total number of vertices is finite, there must be some Brooks cycle $C_p$ that shares a vertex with a preceding Brooks cycle $C_q$ that does not immediately precede $C_p$, so that $q < p-1$. Choose $p$ to be as small as possible so that there exists such a $C_q$ with $q < p-1$ that shares a vertex with $C_p$.  
Also suppose that $C_p$ belongs to $V_k$, for some $k \in A$. Let $v$ be the vertex that is added to $C_p$ to form a complete cycle, and let $u$ be the vertex that is removed from $C_q$ during the $(q+1)$th iteration. Note that $u$ is moved from $C_q$ to cycle $C_{q+1}$ to complete it. We claim that $u \neq v$; if this were not so, then $C_{p-1}$ must contain $u$, in order for $u$ to be moved to $C_p$ during the $p$th iteration. But $p-1 > q+1$, as a vertex cannot be moved out of a cycle immediately after being moved in. Hence $C_{q}$ does not immediately precede $C_{p-1}$, yet they both contain $u$. This contradicts our choice of $p$. 

 Since the maximum degree of the subgraph induced by $V_k$ is 2, $v$ must have been attached to the endpoints of a directed path in $V_k$ to form $C_p$. Since the $p$th iteration is the first one during which two Brooks cycles intersect and neither immediately precedes the other, the exact same directed path must have been left behind when $u$ was removed from $C_q$ during the $q$th iteration of the process.   
 Let the vertices of this directed path be $x_1, \ldots, x_a$, for some $a \geq 2$. 

If $a = 2$, then the vertices $x_1, x_a, u$, and $v$ form an induced subgraph isomorphic to either $G_1$ or $G_2$, depending on whether there is an edge between $u$ and $v$. If $a \geq 3$, then since $C_p$ and $C_q$ are induced, there is no edge between $x_a$ and $x_1$, meaning that the vertices $x_1, x_a, u$, and $v$ form an induced subgraph isomorphic to either $F_1$ or $F_2$. In either case, we have a contradiction to the fact that $D$ contains no induced subgraph isomorphic to any of $F_1, F_2, G_1$, or $G_2$.


Thus there are no Brooks cycles, meaning that for $i \in A$, $D_i$ is acyclic, so $\chi_A(D_i) = 1 = \lceil{\Delta}_i/2\rceil$ for all $i \in A$.
\end{proof}

Our main theorem of this section improves the bound of $\chi_A(D) \leq 2/5 \cdot \Delta(D) + O(1) $ in Corollary $\ref{cor:acyclic}$ to $\chi_A(D) \leq 1/3 \cdot \Delta(D) + O(1)$ for oriented graphs $D$ which do not contain $F_1, F_2, G_1$, or $G_2$ as an induced subgraph.

\begin{theorem}
\label{thm:improvedfg}
Suppose $D$ is an oriented graph which does not contain any of $F_1, F_2, G_1, G_2$ as an induced subgraph.  Then 
\begin{equation}\chi_A(D) \leq \left\lfloor {1}/{3} \cdot {\Delta}(D) + 5/3\right\rfloor.\nonumber
\end{equation}
\end{theorem}

\begin{proof}
We use Lemma \ref{lem:modlov} to show that we can partition the vertices of $D$ into $\Delta(D)/3 + O(1)$ sets, each inducing an acyclic subgraph.

 Set
\begin{equation}
t = \left\lfloor \frac{{\Delta}(D) + 2}{3} \right\rfloor, \quad r = {\Delta}(D) + 2 - 3t\nonumber.
\end{equation} 
Then ${\Delta}(D) = -2 + \left(\sum_{i = 1}^t 3\right) + r$, meaning that if $r \geq 1$, by Lemma \ref{lem:modlov} with $s = t+1$, $\Delta_i = 2$ for $1 \leq i \leq t$, and $\Delta_{t+1} = r-1$, the vertices of $D$ can be partitioned into $t + 1$ sets inducing subgraphs $D_1, \ldots, D_{t+1}$, which satisfy:
\begin{equation}
\chi_A(D_i) \leq 
\begin{cases}
1 \quad \mbox{ if } \quad 1 \leq i \leq t \\
\max\left\{\left\lceil (r-{1})/{2}\right\rceil,1\right\} \quad \mbox{ if } \quad i = t+ 1.\nonumber
\end{cases}
\end{equation}
It is easy to see that $r \leq 2$, so $\lceil (r-1)/2 \rceil \leq 1$. If $r = 0$, then we use Lemma \ref{lem:modlov} with $s = t$ and $\Delta_i = 2$ for $1 \leq i \leq t$, so we have $t$ induced subgraphs $D_1, \ldots, D_t$ with $\chi_A(D_i) = 1$ for each $i$. 
Thus, for all $r \geq 0$, we may give the vertices of each induced subgraph $D_i$ a different color, and the resulting coloring of $D$ has no monochromatic directed cycles, so
\begin{equation}
\chi_A(D) \leq \sum_{i = 1}^{t+1} \chi_A(D_i) \leq t+1 = \left\lfloor \frac{\Delta(D) + 5}{3} \right\rfloor\nonumber.
\end{equation}
\end{proof}

By using Theorem \ref{thm:convertgeo} with $\mathcal{D}$ as the set of all oriented graphs which contain none of $F_1, F_2, G_1$, or $G_2$ as an induced subgraph, we obtain the below corollary of Theorem \ref{thm:improvedfg}.

\begin{corollary}
For any oriented graph $D$ which does not contain any of $F_1, F_2, G_1, G_2$ as an induced subgraph, $\chi_A(D) \leq \lfloor \sqrt{1/2} \cdot \tilde{\Delta}(D) + 5/3 \rfloor.$
\end{corollary}

\section{Concluding remarks}
In this paper, we have proven an upper bound on a generalization of the digraph chromatic number, the $m$-degenerate chromatic number. Moreover, the special case of $m = 1$ gives a bound on the digraph chromatic number which significantly improves previous bounds. However, the bound in Theorem \ref{thm:fracdelta} differs from the conjectured bound in Conjecture \ref{conj:everything} by a factor of $\log {\Delta}(D)$. It seems that a new technique is necessary to obtain the additional factor of $\log {\Delta}(D)$, if it is indeed correct. Moreover, it seems that Conjecture \ref{conj:everything} can be extended to the $m$-degenerate chromatic number.
\begin{conjecture}
There exists a universal constant $c$ such that for each positive integer $m$, every oriented graph $D$ has $\chi_m(D) \leq (c + o(1))\frac{({\Delta}(D)/m)}{\log ({\Delta}(D)/m)}$.
\end{conjecture}


\section{Acknowledgements}
I would like to thank David Rolnick for his helpful discussions and review of the paper. I would also like to thank Jacob Fox for suggesting the direction of research and Tanya Khovanova for helpful suggestions and review. Additionally, I thank anonymous referees for helpful comments, including for suggesting the current version of the proof of Lemma \ref{lem:2separator} and improving the exposition at the beginning of Section 2. I thank Pavel Etingof, David Jerison, and Slava Gerovitch for coordinating the research, and John Rickert for helpful review. Finally I would like to thank the Center for Excellence in Education,  the Research Science Institute and its alumni, and the MIT Math Department for their support.










\end{document}